\documentclass[12pt, reqno]{amsart}
\usepackage{amsfonts,amsmath,amssymb,amsthm,epsfig,cite,graphicx,hyperref,
	color, esint,fancyhdr, enumerate, latexsym,amsrefs, makecell}
\usepackage{color}
\usepackage[mathscr]{eucal}
\usepackage{comment}

\usepackage[nameinlink]{cleveref}

\numberwithin{equation}{section}

\theoremstyle{definition}
\newtheorem{definition}{Definition}[section]

\theoremstyle{remark}
\newtheorem{remark}[definition]{Remark}
\theoremstyle{plain}
\newtheorem{theorem}[definition]{Theorem}

\newtheorem{result}[definition]{Result}

\newtheorem{corollary}[definition]{Corollary}

\newtheorem{conjecture}[definition]{Conjecture}

\newcommand{\Z}{\mathbb{Z}}

\newcommand{\om}{\omega}
\newcommand{\Om}{\Omega}

\newcommand{\st}{\subset}

\newcommand{\D}{\mathbb{D}}

\newcommand{\bdy}{\partial}

\newcommand{\smoo}{\mathcal{C}}

\newcommand{\CC}{\mathbb{C}^2}
\newcommand{\cplx}{\mathbb{C}}

\newcommand{\cn}{\mathbb{C}^n}
\newcommand{\GG}{\mathbb{G}_2}

\begin{document}
\title[weak Wolff-Denjoy theorem for certain non-convex domains]{On weak Wolff--Denjoy theorem for certain non-convex domains}
\author[Vikramjeet Singh Chandel]{Vikramjeet Singh Chandel}
\address{Department of Mathematics and Statistics, Indian Institute of Technology Kanpur,
Kanpur -- 208 016, India}
\email{vschandel@iitk.ac.in}

\author[Sanjoy Chatterjee]{Sanjoy Chatterjee}
\address{Department of Mathematics and Statistics, Indian Institute of Technology Kanpur,
Kanpur -- 208 016, India}
\email{ramvivsar@gmail.com}

\author[Chandan Sur]{Chandan Sur}
\address{Department of Mathematics and Statistics, Indian Institute of Technology Kanpur,
Kanpur -- 208 016, India}
\email{chandans21@iitk.ac.in; survumath@gmail.com }




	\begin{abstract}
	In this paper, we provide a class of domains in $\mathbb{C}^3$, 
    such that every holomorphic self-map of that domain either has a fixed point
    or the sequence of iterates is compactly divergent. 
    In particular, it follows that the
    symmetrized bidisc, symmetrized tridisc, tetrablock, pentablock are in the aforementioned class of domains. We also  prove that the fixed point set of a holomorphic self map of  symmetrized bidisc and tetrablock is either  empty set or a  holomorphic  retract. 
    For the symmetrized bidisc, given a holomorphic self-map such that the sequence of iterates is compactly divergent, we also provide a description of its target set.
	\end{abstract}
     \maketitle
 	\section{Introduction and statements of the results}\label{S:intro}
    Let $\mathbb{D}$ be the open unit disc in the complex plane with center at the origin.
    One of the celebrated results regarding the iteration theory
    of a holomorphic self-map of the unit disc is the Wolff--Denjoy Theorem. 
    \begin{result}\label{R:WD}
    Let $f:\mathbb{D}\to\mathbb{D}$ be a holomorphic self-map. Then the
    following dichotomy holds:
    \begin{itemize}
        \item[$(a)$] $f$ has a fixed point in $\mathbb{D}$. 
        \smallskip

        \item[$(b)$] There exists a point $p\in\bdy\mathbb{D}$ such that 
        the sequence of $n$-th iterates $\{f^n\}$ converges to $p$ uniformly on
        compact subsets of $\mathbb{D}$. 
    \end{itemize}
    \end{result}
    \noindent There has been a lot of interest in finding a generalization of the above theorem
    for a general hyperbolic complex manifold $\mathscr{X}$ equipped with the 
    Kobayashi distance $K_{\mathscr{X}}$. In this direction, Abate has done 
    extensive work for bounded domains $D$ for which the Kobayashi distance 
    $K_D$ is complete; see e.g. \cites{Abate1988Z, Abate1989, Abate1991} 
    and the references therein. In particular, Abate 
    \cite[Theoem 0.6]{Abate1988Z} generalized Result ~\ref{R:WD} for bounded strongly convex domain with $\smoo^{2}$ boundary. Later in \cite[Theorem 1]{Hung94} Huang generalized Result~\ref{R:WD} for contractible, bounded strongly pseudoconvex domain in $\cn$ with $\smoo^{3}$ boundary.
    \smallskip
    
    For Kobayashi hyperbolic domains for which the Kobayashi distance
    is not complete\,---\,which is the generic case
    in higher dimensions\,---\,a new tool, namely, the  visibility property 
    with respect to the Kobayashi distance was introduced by Bharali--Zimmer in 
    \cite{BZ2017} that plays a crucial role in establishing a version of
    Result~\ref{R:WD} for domains that possess this property. 
    \smallskip

    For a complex manifold $\mathscr{X}$ that is taut, given a holomorphic self-map 
    $f$ such that the sequence of iterates $\{f^n\}$ is not compactly divergent, 
    Abate\,---\,inspired by a result of Bedford \cite{Bedford:1983}, proved an important result about the limit points of the sequence  $\{f^n\}$;
    see Result~\ref{Res:exislimman} below. This motivates us to investigate the following three problems on a taut complex manifold:

\begin{itemize}
\item[$(a)$] Describe the asymptotic behavior of the sequence of iterates when it is not compactly divergent.
\smallskip

\item[$(b)$] Determine conditions on the map which ensure that the sequence of iterates is not compactly divergent.
\smallskip

\item[$(c)$] Describe the asymptotic behavior of the sequence of iterates when it is compactly divergent.
\end{itemize}
\noindent Regarding problem $(b)$, Abate in \cite[Theorem 0.4]{Abate1991}, proved that if $\mathscr{X}$ is a taut, acyclic manifold (i.e., $H_{j}(\mathscr{X},\mathbb{Z})=0$ for all $j\in\mathbb{N}$), then for every
$f\in \mathrm{Hol}(\mathscr{X}, \mathscr{X})$ either $f$ admits a periodic point or the sequence $\{f^{k}\}$ is compactly divergent. In the same paper, he conjectured  the following: 

\begin{conjecture}\label{Con:Abate}
   If $\mathscr{X}$ is a taut, acyclic manifold and
   $f\in \mathrm{Hol}(\mathscr{X}, \mathscr{X})$, then $\{f^{k}\}$ is compactly divergent if and only if $f$ has no fixed point. 
\end{conjecture}
\noindent
The conjecture was disproved in \cite{AbateHeinz1992}. However,
Conjecture ~\ref{Con:Abate} is true for bounded convex domains;
see \cite[Theorem~2.4.20]{Abate1989}. Another class for which the above
conjecture holds true is the class of 
taut, acyclic complex manifolds with dimension at most $2$; see Result~\ref{R:Abate} below.
In this article, we are interested 
in the classes of domains in $\mathbb{C}^3$ for which 
the above conjecture holds.
\smallskip

We shall say that a complex manifold $\mathscr{X}$ has the
{\em weak Wolff--Denjoy  property } if the conclusion in Conjecture~\ref{Con:Abate} holds true for every holomorphic self-map of the complex manifold $\mathscr{X}$. Note that the  weak Wolff-Denjoy property is invariant under biholomorphisms.
Therefore, any domain in $\cn$ which is biholomorphic to a bounded convex domain has
the weak Wolff--Denjoy property. In this work, we present
certain special domains\,---\,that are not biholomorphic to a
convex domain\,---\,that possess the weak Wolff--Denjoy property. 
These domains are symmetrized bidisc $\mathbb{G}_{2}$, 
symmetrized tridisc $\mathbb{G}_{3}$, tetrablock $\mathbb{E}$, 
pentablock $\mathcal{P}$; see Section~\ref{S:prelims} for their definition and basic properties that we need in this paper. We now present our first result. 
\begin{theorem}\label{T:WWD}
If $\Omega \in \{\mathbb{G}_{2}, \mathbb{G}_{3}, \mathbb{E}, \mathcal{P}\}$,
then $\Om$ satisfies weak Wolff--Denjoy property.    
\end{theorem}
\noindent 
We present the proof of this theorem in Section~\ref{S:proof}. 
Our proof is based on the concrete description of the
automorphism group ${\rm Aut}(\Om)$ and Result~\ref{R:Abate} below.
The scrupulous reader shall notice that the automorphism group plays an 
important role in the proof of the above theorem. It's natural to
investigate weak Wolff--Denjoy property for domains in
$\mathbb{C}^3$ via its automorphism group. In this direction, we 
first present the following result. 
\begin{theorem}\label{T:Tichotomy}
     Let $\Omega \subset\mathbb{C}^{3}$ be a taut, acyclic domain and $f \in {\rm Hol}(\Om, \Om)$.
     Then one of the following holds:
     \begin{enumerate}
         \item [$(i)$]
         $f$ has a fixed point in $\Om$.
         \smallskip
         
         \item [$(ii)$]
         $\{f^{n}\}$ is compactly divergent.
         \smallskip
         
         \item[$(iii)$]
         $f$ is a periodic automorphism of $\Omega.$
     \end{enumerate}
     \end{theorem}
     \noindent The above theorem is a refinement of a general
     result by Abate \cite[Corollary~2.10]{Abate1991}
     that says that given a taut and acyclic domain 
     $\Om\subset\mathbb{C}^n$, and 
     $f \in {\rm Hol}(\Om, \Om)$ either $\{f^n\}$ is compactly divergent or
     $f$ has a periodic point. 
     \smallskip
     
Now we describe the class of domains having the {weak Wolff--Denjoy property}
by means of their automorphism group.
Recall that for a domain $\Om \subset \cn$ the automorphism
group ${\rm Aut}(\Om)$ forms a topological group under 
usual composition of mapping, and with respect 
to the compact-open topology. Cartan proved that for 
every bounded domain $\Om \subset \cn$,
the group ${\rm Aut}(\Om)$ is a real Lie group.
In what follows, a {\em torus subgroup} of ${\rm Aut}(\Om)$ is
a subgroup that is isomorphic to $\mathbb{T}^r$, for some non-negative integer $r$. 
Here $\mathbb{T}$ is the circle group. We now present our next result. 
\begin{theorem}\label{T:WWD1}
Let $\Omega\subset\mathbb{C}^3$ be a taut, acyclic domain such that every finite cyclic subgroup of $\text{\rm Aut}(\Omega)$ is contained in a torus subgroup of ${\rm Aut}(\Omega)$, then $\Omega$ satisfies weak Wolff--Denjoy property. 
\end{theorem}
\noindent We deduce the following  corollary from Theorem~\ref{T:WWD1}
 \begin{corollary}{\label{col3.9}}
Let $\Omega\subset\mathbb{C}^3$ be a taut, acyclic domain. 
Suppose ${\rm Aut}(\Omega)\cong{\rm Aut}(\mathbb{D})$  then $\Omega$ has
weak Wolff--Denjoy property.
\end{corollary}

We now turn our attention to study the structure of the fixed point set of a
holomorphic self-map of the symmetrized bidisc. Vigu\'{e} \cite{Vigue:1985} 
proved that if $f$ is a 
holomorphic self-map of a convex domain
$\Om \subset \cn$ then the ${\rm Fix} (f):=\{z \in \Om: f(z)=z\}$
is a holomorphic retract of the domain $\Om$, see Section~\ref{S:prelims}
for the definition of a retract.
However, in general, the fixed point set of a holomorphic self-map
may not be connected. 
For example if  $A({1}/{2},\,2):=\{z \in \cplx\,:\,{1}/{2}<|z|<2\}$ and
$f(z)={1}/{z}$. Then $\text{Fix}(f)=\{1,-1\}$. 
For the symmetrized bidisc $\mathbb{G}_{2}$, and the tetrablock $\mathbb{E}$ we have the following result.
\begin{theorem}\label{P:retractofG2}
    Let $f \in {\rm Hol}(\mathbb{G}_{2}, \mathbb{G}_{2})$ be such that ${\rm Fix}(f) \neq \emptyset$.
    Then ${\rm Fix}(f)$ is a holomorphic retract in  $\GG$. In particular, it is connected. 
\end{theorem}
\noindent 

\begin{theorem}\label{T:fix of tetrablock}
      Let $f \in \rm {Hol}(\mathbb{E}, \mathbb{E})$ be such that ${\rm Fix}(f) \neq \phi$. Then ${\rm Fix}(f)$  a holomorphic retract. In particular, it is connected.
\end{theorem}
Let $\Om\subset\cn$ be a bounded domain possessing weak Wolff--Denjoy property. 
Given $f\in{\rm Hol}(\Om, \Om)$ that does not have a fixed point, there is interest 
in locating the {\em target set} $T(f)$ of the map $f$ defined by 
\[
T(f):=\bigcup_{z \in \Om}\{\xi \in \partial \Om:\exists (n_{k})_{k \in 
\mathbb{N}}\,:\,f^{n_{k}}(z) \to \xi, \ \text{as} \ k \to \infty\}.
\]
Herv\'{e} \cite{Herve:1954} studied the set $T(f)$ for a fixed point free 
holomorphic self-map of the bidisc $\D^2$. Later,
Abate--Raissy \cite{Abate_Raissy:2014} studied 
the target set for a general convex domain and, in particular, 
for the polydisc $\D^n$. Recently, Bracci--\"{O}kten \cite{Bracci_Okten:2025}
stated a conjecture regarding the target set of a fixed point
free holomorphic self-map in the polydisc. 
In this paper, we also present a result in the same vein for $\GG$.
\begin{theorem}\label{T:limit in G2}
Let $f \in {\rm Hol}(\mathbb{G}_{2}, \mathbb{G}_{2})$ be a 
fixed point free map. Given $z_0\in \GG$, let 
\[
T(f, z_{0}):=\{\xi \in \partial \Om:\exists (n_{k})_{k \in 
\mathbb{N}}\,:\,f^{n_{k}}(z_0) \to \xi, \ \text{as} \ k \to \infty\}.
\]
Then we have the following:
\begin{enumerate}
 \item [$(i)$]
     If $T(f, z_{0})\subset \{(2e^{i\theta},e^{2i\theta}):\theta \in [0,2\pi)\} 
     \subset\partial\mathbb{G}_{2}$  for some $z_0\in\GG$ then
      $T(f) \subset \{(2e^{i\theta},e^{2i\theta}):\theta \in [0,2\pi)\}$.
      \smallskip
      
        \item [$(ii)$]
        If $T(f,z_{0}) \subset \pi_2(\mathbb{D} \times e^{i\theta_0})$ 
        for some $\theta_0$ then $T(f) \subset \pi_2(\overline{\D} \times e^{i\theta_{0}}) \bigcup \pi_2(e^{i\theta_{0}} \times \overline{\D})$.
    \end{enumerate}
\end{theorem}
\noindent We shall present the proof of Theorem~\ref{P:retractofG2}, Theorem~\ref{T:fix of tetrablock}, Theorem~\ref{T:limit in G2} in Section~\ref{S:target}.
\smallskip

\noindent{\bf Concluding Remarks.}
We do not know whether Conjecture~\ref{Con:Abate} holds for any taut and
acyclic complex manifold 
$\mathscr{X}$ of dimension $3$.  Note that the proof of Theorem~\ref{T:Tichotomy}
also applies to any such complex manifold,
and therefore this leads to the following question:
\begin{itemize}
    \item Let $\mathscr{X}$ be a taut, acyclic complex manifold of dimension $3$. 
    Is it true that every periodic automorphism of such a manifold has a fixed point?
\end{itemize}
Another interesting problem is the following: 
\begin{itemize}
    \item Let $\Om$ be a bounded acyclic $\mathbb{C}$-convex domain. Does 
    $\Om$ satisfy weak Wolff--Denjoy property?
\end{itemize}
Note that $\GG$, $\mathbb{E}$, $\mathcal{P}$ are all $\mathbb{C}$-convex
domains having weak Wolff--Denjoy property. Another related problem is
whether the $n$-th symmetrized polydisc $\mathbb{G}_n$ satisfies 
weak Wolff--Denjoy property.

\section{Preliminaries}\label{S:prelims}
In this section, we recall several definitions
and results from the literature that will be used in the proofs of
our main results. In the first subsection, we recall 
certain special domains that arise in the $\mu$-synthesis problem. 
In the next subsection, we recall definitions and results concerning
the iteration theory of holomorphic self-maps. 
\smallskip

\subsection{Certain special domains}\label{SS:spdo}
We first recall the definition of symmetrized polydisc. 
Let $\pi_{n} : \mathbb{C}^n \to \mathbb{C}^n$
be the symmetrization map defined by
$\pi_{n}(z_1,\dots,z_n)=\bigl(s_1(z), s_2(z), \dots, s_n(z)\bigr)$,
where $s_k(z)$ is the $k$-th elementary symmetric polynomial given by
\[
s_k(z)
=
\sum_{1 \le i_1 < \cdots < i_k \le n}
z_{i_1} \cdots z_{i_k},
\qquad k=1,\dots,n.
\]
The \emph{symmetrized polydisc} $\mathbb{G}_n$ in $\mathbb{C}^n$ is defined by
$
\mathbb{G}_n := \pi_{n}(\mathbb{D}^n)$.
The symmetrized polydisc—particularly the symmetrized bidisc $\GG$—has been
extensively studied in operator theory and several complex variables
due to its function-theoretic significance
(see \cite[Chapter~7]{JPbook2013} and the references
therein for further details on these domains).
It turns out that  the domain $\mathbb{G}_{n}$ is  $(1,2,\cdots ,n)$-balanced domain.
In particular, $\mathbb{G}_{n}$ is contractible.
It is a fact that $\mathbb{G}_{n}$ is complete Kobayashi hyperbolic domain.
We  need the following result due to 
Edigarian--Zwonek about ${\rm Aut}(\mathbb{G}_{n})$.
\begin{result}[paraphrasing {\cite[Theorem~1 ]{EdiZwo2005}}]\label{Res:autgn}
   Let $f: \mathbb{G}_{n} \to \mathbb{G}_{n}$ be a holomorphic map. Then $f \in {\rm Aut}(\mathbb{G}_{n})$ if and only if there exists $h \in {\rm Aut}(\D)$ such that $$f(\pi_{n}(z_{1}, z_{2}, \cdots,z_{n}))=\pi_{n}(h(z_{1}),h(z_{2}), \cdots, h(z_{n})).$$ 
\end{result}

\noindent The following result for $\mathbb{G}_{2}$
is needed in the proof of Theorem~\ref{T:limit in G2}.
\begin{result}\cite[Corollary 2.2]{AglerYoung2004}\label{R:AglerYoung1}
Let $s,p \in \cplx$ then $(s,p) \in \GG$ if and only if $|s| < 2$ and for all 
$\omega \in \mathbb{T}$ we have $|\Phi_{\omega}(s,p)| < 1 $,
where  $$\Phi_{\omega}(s,p):=\frac{2\omega p-s}{2-\omega s}.$$   
\end{result}

\noindent The next result provides a complete description of a complex geodesic in a symmetrized bidisc.
\begin{result}[paraphrasing {\cite[Theorem 2]{ZwoPflug2005}}]\label{Res:compgeosymbi}

A holomorphic mapping $f : \mathbb{D} \to \GG$ is a complex geodesic if and only if it is (up to an automorphism of the unit disc and an automorphism of the symmetrized bidisc) of one of the following two forms:
\begin{align*}
f(\lambda) &= \pi_{2}\big(B(\sqrt{\lambda}),\, B(-\sqrt{\lambda})\big), \quad \lambda \in \mathbb{D},
\end{align*}
where $B$ is a non-constant Blaschke product of degree one or two with $B(0)=0$;
\begin{align*}
f(\lambda) &= \pi_{2}\big(\lambda,\, a(\lambda)\big), \quad \lambda \in \mathbb{D},
\end{align*}
where $a$ is an automorphism of $\mathbb{D}$ such that $a(\lambda) \neq \lambda$ for all $\lambda \in \mathbb{D}$.

Moreover, the complex geodesics in the symmetrized bidisc are unique (up to automorphisms of the unit disc).

\end{result}
We now recall the definition of another important domain arising in 
$\mu$-synthesis problem, namely the \emph{Tetrablock}. The {\em Tetrablock}, denoted by $\mathbb{E}$, is the set
\[
\mathbb{E}
   = \Bigl\{
       (z_1, z_2, z_3) \in \mathbb{C}^3 \,\Bigm|\,
       \bigl| z_2 - \overline{z_1}\,z_3 \bigr|
       + \bigl| z_1 z_2 - z_3 \bigr|
       < 1 - |z_1|^2
     \Bigr\}.
\]
The domain $\mathbb{E}$ is a starlike domain with respect
to the origin (see\cite[Theorem 2.7]{AWY2007}). Hence, it is contractible. It 
is a fact \cite[Corollary ~4.2]{Zwonek2013} that the tetrablock is a $\cplx$-convex domain, hence, it
follows from \cite{Nikolai2007} that $\mathbb{E}$ is a taut domain.
We now give a description of $\text{Aut}(\mathbb{E})$ as presented in 
\cite[Theorem 4.1]{AutoTetrablock} which
is needed in our proof of Theorem~\ref{T:fix of tetrablock}.
Given $\nu, \chi \in {\rm Aut}(\mathbb{D})$, write
$$\nu(z):=\omega \frac{z-\alpha}{\overline{\alpha}z-1}
\ \ \ \text{and} \ \ \ \chi(z):=\sigma \frac{z-\beta}{\overline{\beta}z-1}.$$
Consider $L_{\nu}, R_{\chi} \in {\rm Aut}(\mathbb{E})$ defined as follows:
$L_\nu(x) = \nu \cdot (x_{1},x_{2},x_{3}):=(x_{1}^{'},x_{2}^{'},x_{3}^{'})$ 
such that $$\lambda\begin{bmatrix}
      \omega &   -\omega \alpha \\
      \overline{\alpha}&-1 
\end{bmatrix}\begin{bmatrix} x_{3} &   -x_{1} \\
     x_{2}&-1 
\end{bmatrix} =\begin{bmatrix} x_{3}^{'} &   -x_{1}^{'} \\
     x_{2}^{'}&-1  
\end{bmatrix}$$
where $\lambda \in \cplx$ is chosen so that the $(2,2)$ entry of the
matrix appeared in the left-hand side of the above equation is $-1$. Similarly, 
$R_\chi(x) =  (x_{1},x_{2},x_{3})\cdot\chi :=(x_{1}^{'},x_{2}^{'},x_{3}^{'})$
such that $$\lambda\begin{bmatrix} x_{3} &   -x_{1} \\
     x_{2}&-1 
\end{bmatrix}.\begin{bmatrix}
      \sigma &   -\sigma \beta \\
      \overline{\beta}&-1 
\end{bmatrix} =\begin{bmatrix} x_{3}^{'} &   -x_{1}^{'} \\
     x_{2}^{'}&-1  
\end{bmatrix}$$
where $\lambda \in \cplx$ is chosen such that the $(2,2)$
entry of the matrix appeared in the left-hand side of the above equation is $-1$.
Let $F \in {\rm Aut}(\mathbb{E})$ be
defined by $F(x_{1},x_{2},x_{3})=(x_{2},x_{1},x_{3})$. 
Then
\begin{equation}\label{E:auttetra}
{\rm Aut}(\mathbb{E})=\left\{
L_\nu R_\chi F^{\mu}
:\;
\mu, \chi \in \text{Aut}(\mathbb{D}),
~ \mu \in \{0,1\}
\right\}.\end{equation}
Here, $F^{0}={\rm I}$, and $F^{1}=F$.
\smallskip

We now recall the definition of Pentablock introduced 
by Agler, Lykova and Young \cite{agler2015jmaa} in 2015. 
Here, we mention an equivalent definition of the
Pentablock (see \cite[Theorem 1.1, Theorem 5.2]{agler2015jmaa}).
The {\em Pentablock} is denoted by $\mathcal{P}$ and
is the set
$$ \mathcal{P}:=\left\{(a,s,p) \in \mathbb{C}^{3}:
|a|<\bigg|1-\frac{\frac{1}{2}s\overline{\beta}}{1+\sqrt{1-|\beta|^{2}}}\bigg|, (s,p) \in \mathbb{G}_{2}\right\},$$
where $\beta:={(s-\overline{s}p)}/{(1-|p|)^{2}}$. The domain $\mathcal{P}$ is starlike, $\cplx$-convex but  cannot be exhausted
by domains biholomorphic to convex ones; see \cite{Guicong2020}, \cite{PZapalao2015}. It follows from \cite[Theorem~1]{Nikolai2007} that $\mathcal{P}$ is a taut domain.
\smallskip

\subsection{Iteration Theory}\label{SS:it}
In this subsection, we recall various definitions and certain important results
pertaining to iteration theory of holomorphic self-maps on taut complex manifolds. 
\begin{definition}\label{Def:retract}
A \emph{holomorphic retraction} $\rho$ of a complex manifold $\mathscr{X}$
is a holomorphic map $\rho:\mathscr{X}\to\mathscr{X}$ such that
$\rho \circ \rho = \rho$ on $\mathscr{X}$. The image
of $\rho$ is called a \emph{holomorphic retract} of $\mathscr{X}$.
\end{definition}
Recall that a manifold $M$ is called {\em acyclic} if it is connected and the singular homology group $H_{j}(M, \mathbb{Z})=0$ for all $j >0$.
\begin{remark}\label{rem:Remark1}
   Let $\mathscr{X}$ be a complex manifold and 
   $M \subset \mathscr{X}$ is a holomorphic retract of $\mathscr{X}$.
   Then the diagram 
   $$ M \xrightarrow{\ i\ }  \mathscr{X} \xrightarrow{\ \rho\ } M$$
   induces the following diagram at the level of homology groups
$$
H_j(M,\mathbb{Z}) \xrightarrow{\ i_* \ } H_{j}( \mathscr{X},\mathbb{Z}) \xrightarrow{\ \rho_* \ } H_{j}(M,\mathbb{Z}),
$$
such that $\rho_* \circ i_* = \mathrm{id}_{H_{j}(M,\mathbb{Z})}$.
In particular, if $ \mathscr{X}$ is acyclic, 
then $M$ is also acyclic.
\end{remark}

For a taut complex manifold $ \mathscr{X}$, the following result due to 
Abate is very useful in understanding the behavior of 
$\{f^n\}$, where $f\in{\rm Hol}(\mathscr{X, X})$. 

\begin{result}[paraphrasing {\cite[Theorem~2.1.29]{Abate1989}}]\label{Res:exislimman}
Let $ \mathscr{X}$ be a taut complex manifold,
and $f\in{\rm Hol}( \mathscr{X},\, \mathscr{X})$. 
Assume that the sequence $\{{f^k}\}$ is not compactly divergent.
Then there exists a submanifold $M$ of $ \mathscr{X}$
and a holomorphic retraction $\rho : \mathscr{X} \to M$
such that every limit point $h\in{\rm Hol}( \mathscr{X}, \mathscr{X})$ of $\{ {f^k}\}$ is of the
form $h=\rho \circ \gamma $, where $\gamma$ is an automorphism of $M$.
Moreover, $\rho$ itself is a limit point of the sequence $\{{f^k}\}$.
\end{result}
\begin{remark}\label{Rem:Rem2}
 The manifold $M$ in the above result is called the {\em limit manifold} of the 
map $f$ and $\rho$ is called the {\em limit retraction} of $f$. It also easily follows
from this result that $f|_{M}\in{\rm Aut}(M)$.
We refer the reader to \cite[Section~2.1.3]{Abate1989} for more details.
If  dim$(M)={\rm dim }(\mathscr{X})$ in Result~\ref{Res:exislimman} 
then $M=\mathscr{X}$ and $f \in {\rm Aut}(\mathscr{X})$.
\end{remark}

Given a taut complex manifold $\mathscr{X}$ and
$f\in {\rm Hol}(\mathscr{X}, \mathscr{X})$ such that $\{f^n\}$ is not compactly 
divergent, let us denote by $\Gamma(f)$
the set of all limit points of $\{{f^n}\}$ in ${\rm Hol}(\mathscr{X}, \mathscr{X})$.
We now recall the following two results due to Abate.   
 \begin{result}[paraphrasing {\cite[Theorem 1.2]{Abate1991}}]\label{R:limitgroup}
  Let $\mathscr{X}$ be a taut complex manifold,
  and take $f \in \mathrm{Hol}(\mathscr{X},\mathscr{X})$
  such that the sequence $\{f^{n}\}_{n \in \mathbb{N}}$
  is not compactly divergent. Then there exist integers
  $q(f), r(f) \in \mathbb{N}$ such that
$\Gamma(f) \cong \mathbb{Z}_{q(f)} \times \mathbb{T}^{r(f)} .$
More precisely, $\Gamma(f)$ is isomorphic to the
compact abelian subgroup of $\mathrm{Aut}(M)$
generated by $\varphi = f|_M$, where $M$ is the limit manifold of $f$.
\end{result}

\begin{result}[paraphrasing {\cite[Proposition 2.16]{Abate1991}}]\label{R:limitgroup1}
Let $\mathscr{X}$ be a taut Stein manifold.
Let $f \in \mathrm{Hol}(\mathscr{X},\mathscr{X})$ be
such that $\{f^{n}\}$ is not compactly divergent
such that $\Gamma(f) \cong \mathbb{Z}_{q(f)} \times \mathbb{T}^{r(f)}$.
Let $m(f)$ be the dimension of the limit manifold of $f$.
Then we have 
\begin{enumerate}
\item[$(i)$] $r(f) \leq m(f)$;
\smallskip

\item[$(ii)$] if $\mathscr{X}$ is acyclic and $r(f) = m(f)$,
then $f$ has a fixed point. 
\end{enumerate}
\end{result}
\noindent As mentioned in the introduction, the next result
states that each acyclic, taut complex manifold of 
dimension at most 2 satisfies the weak Wolff--Denjoy Theorem. 
\begin{result}[paraphrasing {\cite[Corollary 2.14]{Abate1991}
}]\label{R:Abate}
Let $\mathscr{X}$ be an acyclic taut complex manifold of dimension at most two
and let $f \in Hol(\mathscr{X}, X)$. Then  $\{f^{k}\}$ is not 
compactly divergent if and only if $f$ has a fixed point in $\mathscr{X}$.
\end{result}

Let $G$ be a group acting on a manifold $M$ via the action $A: G \times M \to M$. 
Then the fixed point set of the action $M^{G}$ is defined as follows:
\[
M^{G}:=\{ x\in M : A(g, x)=x \ \forall g\in G\}.
\]
We shall need the following two results in our proof of Theorem~\ref{T:Tichotomy}. 
\begin{result}[paraphrasing {\cite[Theorem~2.8]{Abate1991}}]\label{R:topo2}
 Let $T$ be a torus group acting smoothly on
 an orientable  manifold $\mathscr{X}$ of finite type. Suppose that $H^{j}(X; \mathbb{Q})=(0)$ for all odd $j$. 
 Then $\mathscr{X}^{T}$ is a nonempty closed (not necessarily connected)
 submanifold of $\mathscr{X}$ of finite topological type. 
\end{result}
\begin{result}[paraphrasing{\cite[Corollary~IV.1.5]{Bredon}}]\label{R:Topo1}
Let $T$ be a torus group acting smoothly on an acyclic manifold $\mathscr{X}$.
Then $\mathscr{X}^{T}$ is acyclic.
\end{result}
\smallskip

\section{Proofs of Theorem~\ref{T:WWD}, Theorem~\ref{T:Tichotomy},
Theorem~\ref{T:WWD1}}\label{S:proof}
We begin this section with the proof of Theorem~\ref{T:WWD}. 

\subsection{The proof of Theorem~\ref{T:WWD}}\label{SS:proofwwd}
\begin{proof}
First we consider the case when $\Om=\GG$. 
As noted in Subsection~\ref{SS:spdo}, $\mathbb{G}_n$, $n\geq 2$, is contractible
and complete Kobayashi hyperbolic, and hence it is acyclic and taut. 
Therefore, by Result~\ref{R:Abate}, $\GG$ has weak Wolff--Denjoy property. 
\smallskip

For the other cases, let $f \in \text{Hol}(\Om,\,\Om)$ be a fixed point free map. 
Assume, to get a contradiction,
that the sequence is not compactly divergent. 
First, note that each $\Om$ is a taut manifold. 
Let  $M$ be the limit  manifold of $f$ and
$\rho :\Om \to M $ be the corresponding retraction as given by Result~\ref{Res:exislimman}.
Since $\text{dim}(\Omega)=3$, so $\text{dim}(M)\leq 3$.
The domain $\mathbb{G}_{3}$
is quasi-balanced and $\mathbb{E}, \mathcal{P}$ are star-shaped.  
Consequently, each $\Omega$ is contractible, hence, it is acyclic.
Therefore, $M$ is taut, and it follows from
Remark~\ref{rem:Remark1} that $M$ is an acyclic submanifold of $\Om$. 
Consider $\phi=f|_{M}$ and recall that $\phi\in{\rm Aut}(M)$. 
Two cases arise:
first suppose $\dim{M}\leq 2$. 
In this case, it follows from Result~\ref{R:Abate} the sequence $\{\phi^n\}$ is compactly 
divergent, a contradiction to our assumption.
\smallskip

Now suppose $\text{dim}(M)=3$. In this case, it 
follows from Remark~\ref{Rem:Rem2} that
$M=\Om$ and $f \in \text{Aut}(\Omega)$.
We now consider the three cases:
\smallskip

\noindent\textbf{Case 1.} $\Omega:=\mathbb{E}$. 
\smallskip

\noindent It is a fact that  the triangular set 
$\mathcal{T}:=\{(a,b,ab): a, b \in \mathbb{D}\}\subset\mathbb{E}$ is invariant under 
every automorphism of $\mathbb{E}$, see \cite[Remark~6.6]{AWY2007}.
Note that the set $\mathcal{T}$ is biholomorphic to the bidisc. 
Let $\psi: \mathbb{D}^{2} \to \mathcal{T}$ be a biholomorphic embedding.
Then $\psi^{-1}\circ  f \circ \psi : \mathbb{D}^{2} \to \mathbb{D}^{2}$ is
a biholomorphism that does not have a fixed point.
Since $\mathbb{D}^2$ is convex, 
the sequence $\{g^n\}$ is compactly divergent where 
$g=\psi^{-1}\circ  f \circ \psi$. But this
implies that the sequence of iterates of $f|_{\mathcal{T}}$ is 
compactly divergent. This is a contradiction to our
assumption.
\smallskip

\noindent \textbf{Case~2.} $\Omega:=\mathbb{G}_{3}$. 
\smallskip

\noindent As $f \in \text{Aut}(\mathbb{G}_{3})$, by \cite[Theorem~1]{EdiZwo2005} 
there exists $h \in \text{Aut}(\mathbb{D})$ such that
\begin{align} \label{E:autotio}
f(\pi_3(z_{1},z_{2},z_{3}))=\pi_3(h(z_{1}),h(z_{2}),h(z_{3})) ~~\forall z_{1}, z_{2}, z_{3} \in \mathbb{D}.
\end{align}
Here $\pi_3:\mathbb{C}^3\to\mathbb{C}^3$ is the symmetrization map
as defined in Subsection~\ref{SS:spdo}. 
Observe that the map $h \in \text{Aut}(\mathbb{D})$ cannot have a fixed point in $\mathbb{D}$, 
otherwise, $f$ will have a fixed point in $\mathbb{G}_{3}$.
Appealing to the weak Denjoy--Wolff property for $\mathbb{D}$, 
we conclude that $\{h^{n}\}$ is compactly divergent on 
$\mathbb{D}$.
It follows from \eqref{E:autotio} that 
$$f^{n}(\pi(z_{1},z_{2},z_{3}))=\pi(h^{n}(z_{1}), h^{n}(z_{2}), h^{n}(z_{3}))$$ for
all $n \in \mathbb{N}$ for  all $z_{1},z_{2}, z_{3} \in \mathbb{D}$.
This implies that $\{f^{n}\}$ is compactly divergent on $\mathbb{G}_{3}$ and we are done in this case too.
\smallskip

\noindent\textbf{Case~3.} $\Om:=\mathcal{P}$. 
\smallskip

\noindent By \cite[Theorem~15]{Kosiski2015}, we know that each element in 
$\text{Aut}(\mathcal{P})$ is of the following form 
\begin{align}
f_{\omega,\,\gamma} (a ,\lambda_{1}+\lambda_{2}, \lambda_{1}\lambda_{2}) =
\bigg(\frac{\omega(1-|\alpha|^2 )a}{ 1-\overline{\alpha}(\lambda_{1}+\lambda_{2})+\overline{\alpha}^{2}\lambda _{1} \lambda_{2}}
, \gamma(\lambda_{1} ) + \gamma(\lambda_{2} ), \gamma(\lambda_{1} )\gamma(\lambda_{2} )\bigg)
\end{align}
\noindent
where $(a, \lambda_{1} + \lambda_{2}, \lambda_{1} \lambda_{2} ) \in \mathcal{P}$,
$\lambda_{1},   \lambda_{2} \in  \mathbb{D}$, $\gamma\in{\rm Aut}(\mathbb{D})$
and $\omega \in  \mathbb{T}$, $\alpha =\gamma^{-1}(0)\in  \mathbb{D}$.
Therefore, we can assume that the map $f=f_{\omega,\,\gamma}$
for some choice of
$\omega\in\mathbb{T}$ and $\gamma\in{\rm Aut}(\mathbb{D})$.
We claim that $\gamma$ does not
have a fixed point in $\mathbb{D}$, for if $z_0\in\mathbb{D}$ is a fixed point,
then it follows that  $f_{\om, \gamma}(0, 2z_{0}, z_{0}^{2})=(0, 2z_{0}, z_{0}^{2})$. 
However, by our assumption, $f$ does not have a fixed point. It follows, as before, that
the sequence $\{\gamma^n\}$ is compactly divergent, in fact,
we know that $\gamma^n\to\omega_{0} \in \mathbb{T}$ as $n \to \infty$. 
Consequently,  $f_{\om, \gamma}^{n}(0,0,0)=(0, 2\gamma^{n}(0),(\gamma^{n}(0))^{2})$.
Since $(2\gamma^{n}(0),(\gamma^{n}(0))^{2}) \to \partial_{s}\mathbb{\mathbb{G}}_{2}$,
it follows that $f^n_{\omega,\gamma} (0) \to \partial_{s}\mathcal{P}$.
In particular, $\{f^n\}$ is compactly divergent, a contradiction to our assumption. 
\smallskip

This concludes the proof of the theorem. 
\end{proof}

\subsection{Proofs of Theorem~\ref{T:Tichotomy},
Theorem~\ref{T:WWD1} and Corollary~\ref{col3.9}}. 
In this subsection, we present the proofs of Theorem~\ref{T:Tichotomy},
Theorem~\ref{T:WWD1} and Corollary~\ref{col3.9}. 

\begin{proof}[Proof of Theorem~\ref{T:Tichotomy}]
    Let $f\in\rm Hol(\Omega, \Omega)$ be such that
    $\rm{Fix}({f}) = \emptyset$ and $\{f^n\}$ is not compactly divergent.
    By Result~\ref{Res:exislimman},
    there exists a holomorphic retract $M$ of $\Omega$
    such that $f|_{M}\in{\rm Aut}(M)$. 
    As noted in Remark~\ref{rem:Remark1}, 
    $M$ is acyclic and taut.
    Observe if ${\rm dim}(M)\leq 2$, then by Result~\ref{R:Abate},
    we deduce that $f|_{M}$ has either a fixed point or it is compactly divergent.
    This leads to a contradiction to our assumption.
    Therefore, ${\rm dim}(M)=3$. As noted in Remark~\ref{Rem:Rem2}, in this case 
    we have $M=\Om$ and $f\in $\text{Aut}$(\Omega)$. Appealing 
    to Result~\ref{R:limitgroup} and Result~\ref{R:limitgroup1},
    it follows that
    $\Gamma(f)\leq \text{Aut}(\Omega)$ is isomorphic to 
    $\mathbb{Z}_p\times\mathbb{T}^r$
    for some nonnegative integer $p$ and $r \in \{0,1,2\}$.
    Let $h: \Gamma(f)\to \mathbb{Z}_p\times\mathbb{T}^r $ be a group isomorphism.
    We  consider the following cases:
    \smallskip
    
\noindent\textbf{Case 1.} $r=0$. 
\smallskip

\noindent In this case, if $h(f)=\overline{j}$, then  $h(f^{p})=\overline{0}$ whence
$f^p={\rm I}$ and we are done.
\smallskip

\noindent\textbf{Case 2.} $r=2$.  
\smallskip

\noindent In this case, suppose 
$h(f)=(\overline{j}, e^{2\pi i \theta_1} ,e^{2\pi i \theta_2})$
for some $\overline{j}\in\Z_{p}$ and $\theta_{j} \in [0,1)$ for $j=1,2$.
If $\theta_{j} \in \mathbb{Q}$  with $\theta_j={p_j}/{q_j}$
for $j \in \{1,2\}$, then  we have 
\begin{align*}
    (h(f))^{pq}=(\overline{j} , (e^{2\pi i\theta_1} , e^{2\pi i\theta_2}))^{pq}=(0,(1,1))=h(f^{pq}),
\end{align*}
where $q=\text{l.c.m} (q_{1}, q_{2})$. This implies that $f^{pq}={\rm I}$,
i.e.,  ${f}$ is a periodic  automorphism of $\Om$. Therefore, we are done
when $\theta_j\in\mathbb{Q}$, $j=1,2$. 
\smallskip

For the other case, we make the following claim.
\smallskip

\noindent{\bf Claim.}
Suppose $\theta_{j}\in [0,1) \setminus \mathbb{Q} $ for some $j=1,2$,
then there exists $l \in \mathbb{N}$ such that
${\rm Fix}(f^{l})= \Omega^{T}$ where $T\leq \text{Aut}(\Omega)$ is either
isomorphic to $\mathbb{T}$ or $\mathbb{T}^2$.
\smallskip

\noindent
To prove the claim, we consider two
cases. First assume that  $\theta_{1}={p_1}/{q_1} \in \mathbb{Q}$ and
$\theta_2\in {\mathbb{R}}\setminus{\mathbb{Q}}$.
In this case, if  $h(f)$ = $(\overline{j}, e^{2\pi i \theta_1} ,e^{2\pi i \theta_2})$,
then $h(f^{pq_1}) = (\overline{0}, 1, e^{2\pi i\beta})$, 
where $\beta$ is the fractional part of $ p q_{1} \theta_{2}$.
Since $\theta_2\in \mathbb{R}\setminus\mathbb{Q}$, 
$\beta$ is also an irrational number. 
Therefore, we have 
$$
\overline{\left\{e^{2\pi i\beta n}: n \in \mathbb{N}\right\}}=\mathbb{T}.
$$
In this case, we shall show that 
$l=pq_1$ and $T=h^{-1}(\overline{0}\times 1\times\mathbb{T})$. 
To see that, let $x\in \text{Fix}(f^{pq_1})$
and $(\overline{0}, 1, e^{2\pi i \alpha})
\in \{\overline{0}\}\times \{1\}\times\mathbb{T}$
for some $\alpha \in [0,1)$. Then there exists a sequence
$\{m_{k}\}_{k \in \mathbb{N}}$ such that 
$(\overline{0}, 1, e^{2\pi i  m_{k} \beta })
\to (\overline{0}, 1, e^{2\pi i \alpha})$.
Therefore, we deduce that  
    \begin{align*}
       h^{-1}(\overline{0}, 1, e^{2\pi i \alpha})(x)
       &=h^{-1}(\lim_{k \to \infty}(\overline{0}, 1, e^{2\pi i  m_{k} \beta }))(x) \notag\\
       &=\lim_{k \to \infty}(h^{-1}(\overline{0}, 1, e^{2\pi i \beta }))^{m_{k}}(x)\notag \\
       &=\lim_{k \to \infty}(f^{pq_{1}})^{m_{k}}(x)=x.
    \end{align*}
    This implies that $x\in\Omega^{T}$,
    where $T=h^{-1}(\overline{0}\times 1\times\mathbb{T})$, 
    whence it follows that
    $\text{Fix}(f^{pq_1})\subset\Omega^{T}$. 
    Conversely, if $x\in\Omega^{T}$ then 
    by definition $g\cdot x=x$ for all $g\in {T}$. 
    In particular,
    $h^{-1}(\overline{0}, 1, e^{2\pi i \beta})\cdot x=x$, i.e.,
    $f^{pq_1}(x)=x$ whence $x\in {\rm Fix}(f^{pq_1})$. This also implies that
    $\Om^T\subset {\rm Fix}(f^{pq_1})$, and we are done in this case. 
    The case $\theta_1 \in \mathbb{R}\setminus \mathbb{Q}$ and $\theta_2\in \mathbb{Q}$
    can be dealt in a similar fashion and in this case
    $T=h^{-1}(\overline{0}\times\mathbb{T}\times{1})$.
    \smallskip

    Now we consider the case when
    $\theta_{1}, \theta_{2}\in \mathbb{R}\setminus\mathbb{Q}$ and 
    $h(f)=(\overline{j}, e^{2\pi i\theta_1}, e^{2\pi i\theta_2})$.
    Then $h(f^{p})=(\overline{0}, e^{2\pi i \beta_{1}}, e^{2\pi i \beta_{2}})$,
    where $\beta_{j} \in \mathbb{R}\setminus \mathbb{Q}$ is
    the fractional part of $p\theta_{j}$ for $j=1,2$.
    We deal with this case by considering two subcases.
    First  we assume that $1, \beta_1, \beta_2$ are linearly
    independent over $\mathbb{Q}$, and show that $l=p$ and 
    $T= h^{-1}(\overline{0} \times \mathbb{T}^{2})$.
    From Kronecker--Weyl's Theorem (see \cite{KRONECKERwils}) we have 
  \begin{align}\label{E:eq1}
    \overline{\{{(e^{2\pi i\nu\beta_1 },e^{2\pi i \nu\beta_2 })}: \nu\in\mathbb{N}\}}=\mathbb{T}^2.
    \end{align} 
   Let $x \in \text{Fix}(f^{p})$ and
   $h^{-1}((\overline{0},e^{2\pi i\alpha_1 },
   e^{2\pi i \alpha_2 })) \in h^{-1} (\overline{0} \times\mathbb{T}^{2})$.
   From \eqref{E:eq1} we conclude that there exists
   a sequence $\nu_{k} \in \mathbb{N}$ such that
   $(e^{2\pi i\nu_{k}\beta_1 },e^{2\pi i \nu_{k}\beta_2 })
   \to (e^{2\pi i\alpha_1 }, e^{2\pi i \alpha_2 })$ as $k \to \infty$. 
   Therefore, we have 
\begin{align*}
      h^{-1}(\overline{0}, e^{2\pi i\alpha_1 }, e^{2\pi i \alpha_2 })(x)&= h^{-1}(\lim_{k \to \infty}(\overline{0}, e^{2\pi i\nu_{k}\beta_1 },e^{2\pi i \nu_{k}\beta_2 }))(x)\\
      &=\lim_{k \to \infty}h^{-1}(\overline{0}, e^{2\pi i\nu_{k}\beta_1 },e^{2\pi i \nu_{k}\beta_2 })(x)\\
      &=\lim_{k \to \infty}(h^{-1}(\overline{0}, e^{2\pi i \beta_{1}},e^{2\pi i \beta_{2}}))^{\nu_{k}}(x)\\
      &=\lim_{k\to\infty}f^{p\nu_k}(x)=x.
   \end{align*}
 Therefore, ${\rm Fix}(f^{p}) \subset \Om^{T}$ where $T=h^{-1}(\overline{0} \times \mathbb{T}^{2})$. 
 Conversely, if $x\in\Omega^{T}$ then 
    by definition $g\cdot x=x$ for all $g\in {T}$. 
    In particular,
    $h^{-1}(\overline{0}, e^{2\pi i \beta_{1}}, e^{2\pi i \beta_{2}})\cdot x=x$, i.e.,
    $f^{p}(x)=x$ whence $x\in {\rm Fix}(f^{p})$. Consequently,
    $\Om^T\subset {\rm Fix}(f^{p})$, and we are done in this subcase.
     We now consider the subcase that   $\beta_1, \beta_2$,
    $1$ is linearly dependent over $\mathbb{Q}$.
    Observe that if 
    $$
    G:=\overline{\{(e^{2\pi i \nu \beta_{1}},e^{2 \pi i \nu \beta_{2}}): \nu \in  \mathbb{Z}}\},
    $$
    then $\{\overline{0}\} \times G \leq \{\overline{0}\} \times \mathbb{T}^{2}$.
    Clearly $\{\overline{0}\} \times G $ is a closed subgroup of $\{\overline{0}\} \times \mathbb{T}^{2}$.
    Since both $\beta_{1}, \beta_{2} $ are irrational numbers with
    dim$_{\mathbb{Q}}\{\beta_1, \beta_2, 1\}=2$, 
    there exist $a, b,c \in \mathbb{Z}$ such that not all of them are zero,   
    and $a\beta_{1}+b \beta_{2}+c=0$. Hence, $G \leq \{(z_{1}, z_{2}) \in \mathbb{T}^{2}:z_{1}^az_{2}^b=1\} \cong \mathbb{T}$.
    Hence, the group $G$ can be thought of as an infinite, compact subgroup of the group $\mathbb{T}$. Consequently, $G\cong \mathbb{T}$. We show that in this subcase $l=p$ and $T=h^{-1}(\overline{0}\times G) \cong \mathbb{T}$. 
    To see this, let $x \in {\rm Fix }(f^{p})$ and $h^{-1}(\overline{0}, g) \in  h^{-1}(\overline{0}, G).$ Note that $f^{-p}(x)=(f^{p})^{-1}x=x$. Consequently, $f^{\nu p}(x)=x$ for all $\nu \in \mathbb{Z}$.
    Let $(\nu_k)\subset\mathbb{Z}$ be such that 
    $(\overline{0},e^{2\pi i \nu_{k}\beta_{1}},e^{2 \pi i \nu_{k} \beta_{2}})\to
    (\overline{0}, g)$. Therefore, we deduce that 
\begin{align*}
  h^{-1}(\overline{0}, g)(x)&=h^{-1}(\lim_{\nu_{k} \to \infty}(\overline{0},e^{2\pi i \nu_{k}\beta_{1}},e^{2 \pi i \nu_{k} \beta_{2}}))(x) \\
  &=\lim_{k \to \infty}(h^{-1}(0,e^{2 \pi i \beta_{1}},e^{2 \pi i \beta_{2}}))^{\nu_{k}}(x)\\
  &=\lim_{k \to \infty}f^{\nu_{k}p}(x)=x.
\end{align*}
This implies that ${\rm Fix}(f^{p}) \subset \Omega^{T}$.
Conversely, if  $x \in \Omega^{T}$ then $g.x=x$ for all $g \in T$.
In particular, 
$h^{-1}(0, e^{2\pi i \beta_{1}}, e^{2\pi i \beta_{2}})(x)=f^{p}(x)=x$.
This implies that $x\in{\rm Fix}(f^p)$, and hence 
$\Om^T\subset {\rm Fix}(f^{p})$.
This establishes
the claim.\hfill $\blacktriangleleft$

\smallskip

 Here $\Omega$ is an acyclic, taut manifold and $T$ is a  torus group.
 Therefore, from Result~\ref{R:topo2} we conclude that
 $\Om^{T}$ is non-empty. 
 From Result ~\ref{R:Topo1}, we conclude that $\Omega^{T}$
 is acyclic, in particular, it is connected.
 Appealing to a result due to Vigu\'e \cite{Vigue1986},
 we conclude that $\text{Fix}(f^{l})$
 is a closed complex submanifold of $\Omega$, and hence it is taut.
 Therefore, $\text{Fix}(f^{l})=\Omega^{T}$ is a non-empty,
 closed, acyclic, taut submanifold for all $l \in \mathbb{N}$ considered
 in the above claim.
Note that, $f(\text{Fix} (f^{l}))\subset \text{Fix}(f^{l})$.
Therefore, if ${\rm dim}(\Omega^{T})\leq 2$,
then again from Result~\ref{R:Abate} either
$f$ has a fixed point or $(f^n)$ is compactly divergent.
This contradicts our assumption.
Hence, ${\rm dim}(\text{Fix}(f^{l}))=3$, 
i.e., $\text{Fix}(f^{l})=\Omega$. 
Consequently, $f$ is a periodic automorphism whence we
are done in this case. 
\smallskip

The case $r=1$ can be dealt in a similar way.
This establishes the theorem.
\end{proof}

We now present the proof of Theorem~\ref{T:WWD1}. 

\begin{proof}[Proof of Theorem~\ref{T:WWD1}]
 We show that every periodic automorphism has a fixed point.
 Let $f\in\text{Aut}(\Omega)$ be a periodic automorphism of period $p\geq 2$. 
 Then $\{I_{\Omega},f,f^2,...,f^{p-1}\}$ is a finite
 cyclic subgroup of ${\rm Aut}(\Omega)$. Now from the hypothesis,
 there exists a torus subgroup $T$ of ${\rm Aut}(\Om)$ 
 such that 
 $$
 \{{\rm I}_{\Om},f,f^2,\cdots,f^{p-1}\}\leq T \leq\text{Aut}(\Omega).
 $$
 Therefore,  the action $\mathbb{Z}_p\times\Omega\rightarrow\Omega$
 defined by $(\overline{j}, x)\rightarrow f^j(x)$ has an extension to the action
 $T \times\Omega \rightarrow\Omega$ . Now appealing to Result~\ref{R:topo2},
 the action $ T \times\Omega\rightarrow\Omega$
 has a fixed point, i.e., $\Omega^{T}\neq\emptyset$,
 in particular $f^j(x)=x , \forall j\in\{0,1,...,,p-1\}$.
 Hence, the map  $f$ has fixed point in $\Om$.
\end{proof}

\begin{proof}[Proof of Corollary\ref{col3.9}]
    Let $H \leq \text{Aut}(\Omega)$ be a finite cyclic subgroup.
    Then  $\psi(H)\leq\text{Aut}(\mathbb{D})$ is a finite cyclic subgroup,
    where $\psi : \text{Aut}(D)\rightarrow\text{Aut}(\mathbb{D})$ is a group
    isomorphism. Since every finite cyclic subgroup of ${\rm Aut}(\D)$ is
    contained in a circle subgroup of ${\rm Aut}(\D)$,
    invoking Theorem~\ref{T:WWD} gives the result.
    \end{proof}

\begin{remark}\label{Rem:autg3}
    By Agler's result $\text{Aut}(\mathbb{G}_3)\cong\text{Aut}
    (\mathbb{D})$. $\mathbb{G}_3$ is taut and acylic. Hence,
    it follows from the corollary above that $\mathbb{G}_3$ has 
    weak Wolff--Denjoy property.
\end{remark}

\section{Fixed point set and Target set}\label{S:target}
In this section we present the proofs of our results related to the study of 
the fixed point set of holomorphic self-map of $\mathbb{G}_{2}$
and $\mathbb{E}$, namely 
Theorem~\ref{P:retractofG2} and Theorem~\ref{T:fix of tetrablock} 
respectively. We also present the 
proof of our result about the target set of a fixed 
point free self-map of $\GG$, namely Theorem~\ref{T:limit in G2}.
\subsection{Proof of Theorem~\ref{P:retractofG2}}\label{SS:fixsymbi}

\begin{proof}[Proof of Theorem~\ref{P:retractofG2}]
  Let $f \in \text{Hol}(\mathbb{G}_{2}, \mathbb{G}_{2})$
  be such that $\text{Fix}(f) \neq \emptyset$. 
  Clearly $\{f^{n}\}$ is not compactly divergent.
  Let $M$ be the limit manifold of $f$ and
  $\rho:\mathbb{G}_{2}\to M$ be 
  the holomorphic retraction given by Result~\ref{Res:exislimman}.
  Since $M$ is a holomorphic retract of $\mathbb{G}_{2}$,
  hence, $M$ is a closed, simply connected, noncompact,
  hyperbolic,  complex submanifold of $\mathbb{G}_{2}$.
  Moreover, since 
  $\rho$ itself is a limit point of $\{f^n\}$, we have
  $\text{Fix}(f) \st M$. We now consider
  the possible cases of ${\rm dim}(M)$.
  \smallskip
  
  \noindent\textbf{Case 1.} ${\rm dim}(M)=0$. 
  \smallskip
  
  \noindent Then $M$ is singleton set.
  Consequently, $\text{Fix}(f) $ is also a singleton,
  hence a retract.
  \smallskip

  \noindent\textbf{Case 2.} ${\rm dim}(M)=1$. 
  \smallskip
  
  \noindent If $\text{dim}(M)=1$, then by
  the uniformization theorem we conclude that $M$
  is biholomorphic to $\mathbb{D}$.
  Let $\varphi: \mathbb{D} \to M $ be a biholomorphism.
  It follows from Result~\ref{Res:exislimman} that
  $f|_{M} \in \text{Aut}(M)$. Therefore, $\varphi^{-1}\circ f\circ \varphi: \mathbb{D} \to \mathbb{D}$ is an automorphism.
  Indeed, the map $\varphi^{-1}\circ f\circ \varphi$
  either has a unique fixed point in $\mathbb{D}$ or
  it is identically equal to the identity automorphism
  of $\mathbb{D}$.
Note that  $\varphi^{-1}({\rm Fix}(f))={\rm Fix}(\varphi^{-1}\circ f\circ \varphi)$. Consequently, it follows that either
  ${\rm Fix}(f)$ is singleton or it is equal to  $M$.
  In both the cases ${\rm Fix}(f)$ is a holomorphic retract,
  and in particular, it is connected.
  \smallskip

\noindent\textbf{Case 3.} ${\rm dim}(M)=2$. 

\noindent Then $M=\mathbb{G}_{2}$, and as noted in Remark~\ref{Rem:Rem2},
$f \in \text{Aut}(\mathbb{G}_{2})$.
By Result~\ref{Res:autgn}, 
there exists an $h \in \text{Aut}(\mathbb{D})$ 
such that
\begin{align}\label{E:autobi}
   f(s,p)&=(h(z_{1})+h(z_{2}), h(z_{1})h(z_{2}))
\end{align}  
for all $(s,p)=(z_{1}+z_{2},z_{1}z_{2})$
with $(z_{1},z_{2}) \in \mathbb{D}^{2}$. This case is dealt with the following subcases:
\smallskip

\noindent\textbf{Subcase 3A: $h={\rm I}_{\D}$.}
\smallskip

\noindent Then $f={\rm I}_{\GG}$, so ${\rm Fix}(f)=\GG$, which is trivially a holomorphic retract of itself.
\smallskip

\noindent\textbf{Subcase 3B: $h^2 \neq {\rm I}_{\D}$.}
\smallskip

\noindent Let $z_{1},z_{2} \in \D$ such that $\pi_2(z_1,z_2)\in{\rm Fix}(f)$. Then
\begin{align}\label{E:fixeuality}
f(\pi_2(z_1,z_2))=\pi_2\bigl(h(z_1),h(z_2)\bigr)=\pi_2(z_1,z_2)   
\end{align}
\noindent
It follows from the above relation that either $h(z_{j})=z_{j}$ for $j \in \{1,2\}$ or $h(z_{1})=z_{2}$ and $h(z_{2})=z_{1}$. 
In the first alternative, since $h\neq{\rm I}_{\D}$, and therefore $h$ has at most one fixed point in $\D$, we have $z_1=z_2=\lambda_0$, 
where $\lambda_0$ is the unique fixed point of $h$ in $\D$.
In the second alternative,
\[
h^2(z_1)=z_1,\qquad h^2(z_2)=z_2.
\]
Since $h^2\neq{\rm I}_{\D}$, the automorphism $h^2$ has at most one fixed point in $\D$. Thus $z_1=z_2$.  Consequently, $h(z_1)=z_1$, and hence again $z_1=z_2=\lambda_0$, where again $\lambda_{0}$ is the unique fixed point of $h$ in $\D$.
Therefore, it follows that 
\[
{\rm Fix}(f)=\bigl\{\pi_2(\lambda_0,\lambda_0)\bigr\}
       =\bigl\{(2\lambda_0,\lambda_0^2)\bigr\}.
\]
Thus ${\rm Fix}(f)$ is a singleton and hence a holomorphic retract of $\GG$.
\smallskip

\noindent\textbf{Subcase 3C: $h^2={\rm I}_{\D}$ and $h\neq{\rm I}_{\D}$.}
\smallskip

\noindent In this case, $h$ has a unique fixed point
$\lambda_0\in\D$, namely
the midpoint of the unique geodesic segment connecting $0$ and $h(0)$. 
We now determine the fixed-point set of $f$. 
If $z\in\D$, then
\[
f\bigl(\pi_2(z,h(z))\bigr)
 =\pi_2\bigl(h(z),h^2(z)\bigr)
 =\pi_2\bigl(h(z),z\bigr)
 =\pi_2\bigl(z,h(z)\bigr).
\]
Therefore, 
$
\bigl\{\pi_2(z,h(z)):z\in\D\bigr\}\subset {\rm Fix}(f).
$
Conversely, if $\pi_2(z_1,z_2)\in \rm Fix(f)$, 
again from the relation \eqref{E:fixeuality} we get  either $h(z_{j})=z_{j}$ for $j \in \{1,2\}$ or
$h(z_{1})=z_{2}$ and $h(z_{2})=z_{1}$. 
In the first case, the uniqueness of the fixed point of
$h$ gives $z_1=z_2=\lambda_0$. Clearly 
$(2\lambda_{0}, \lambda_{0}^2) \in \{\pi_{2}(z, h(z)): z\in \D\}$.
In the second case, $z_2=h(z_1)$, and
hence $\pi_{2}(z_{1},z_{2})=(z_{1}+h(z_{1}),z_{1}h(z_{1})).$ Consequently,
\begin{align}\label{E:set1}
{\rm Fix}(f)=\bigl\{\pi_2(z,h(z)):z\in\D\bigr\}.
\end{align}
Note that, since $h(\lambda_0)=\lambda_0$, the point
$
\pi_2(\lambda_0,h(\lambda_0))
 =\pi_2(\lambda_0,\lambda_0)
 =(2\lambda_0,\lambda_0^2)
$
belongs to $\rm{Fix}(f)\cap \mathcal{R}$, where
$
\mathcal{R}:=\bigl\{\pi_2(\lambda,\lambda):\lambda\in\D\bigr\}
       =\bigl\{(2\lambda,\lambda^2):\lambda\in\D\bigr\}
$
denotes the royal variety.
Indeed, if $\pi_2(z,h(z))\in\mathcal{R}$,
then $z=h(z)$, and therefore $z=\lambda_0$.
Thus
$
\rm Fix(f)\cap\mathcal{R} 
 =\bigl\{(2\lambda_0,\lambda_0^2)\bigr\}.
$
\smallskip

We now prove that the set in~\eqref{E:set1} is
a holomorphic retract. Choose $b\in \rm{Aut}(\D)$ such
 that $b(\lambda_0)=0$, and let $A\in \rm{Aut}(\GG)$
 be the automorphism induced by $b$:
\[
A\bigl(\pi_2(z_1,z_2)\bigr)
 =\pi_2\bigl(b(z_1),b(z_2)\bigr).
\]
Set
$
\widetilde f:=A\circ f\circ A^{-1}.
$
Then $\widetilde f$ is induced by
$
\widetilde h:=b\circ h\circ b^{-1}.
$
The automorphism $\widetilde h$ is a nontrivial involution
fixing the origin. Hence it is a rotation
$\widetilde h(z)=\omega z$ with $\omega^2=1$ and $\omega\neq1$.
Therefore
$
\widetilde h(z)=-z.
$
It follows that
$
\widetilde f(s,p)=(-s,p).
$
Consequently,
\[
{\rm Fix}(\widetilde f)
 =\{(s,p)\in\GG:s=0\}
 =\{(0,p):p\in\D\}.
\tag{3.5}\label{eq:fix-tilde}
\]
To see the last equality, note that if $(0,p)=\pi_2(z_1,z_2)\in\GG$,
then $p=z_1z_2\in\D$; conversely, for every $p\in\D$
one may choose $z\in\D$ with $z^2=-p$, and then $(0,p)=\pi_2(z,-z)\in\GG$.
Define
\[
r:\GG\longrightarrow\GG,
\qquad r(s,p)=(0,p).
\]
For $(s,p)\in\GG$ we have $p\in\D$, and hence $(0,p)\in\GG$ by~\eqref{eq:fix-tilde}. Thus $r$ is well defined and holomorphic. Moreover,
$
r\circ r=r$ and 
$r(\GG)=\rm Fix(\widetilde f)$.
Therefore $r$ is a holomorphic retraction
onto $\rm Fix(\widetilde f)$.
Since $\widetilde f=A\circ f\circ A^{-1}$,
$\rm Fix(\widetilde f)=A\bigl(\rm Fix(f)\bigr)$,
and hence $\rm Fix(f)=A^{-1}\bigl(\rm Fix(\widetilde f)\bigr).$
Consequently,
\[
R:=A^{-1}\circ r\circ A
\]
is a holomorphic retraction of $\GG$ onto ${\rm Fix}(f)$.
\smallskip

Thus, in every case, ${\rm Fix}(f)$ is
a holomorphic retract of $\GG$. Since the
image of the connected domain $\GG$ under a retraction is connected, ${\rm Fix}(f)$ is connected as well.
\end{proof}
\smallskip

\subsection{Proof of Theorem~\ref{T:fix of tetrablock}}\label{SS:fixtetra}
We now present the proof of Theorem~\ref{T:fix of tetrablock}.
\begin{proof}[Proof of Theorem~\ref{T:fix of tetrablock}]
    Let $f\in{\rm Hol}(\mathbb{E},\mathbb{E})$ be such that $\text{Fix}(f) \neq \emptyset$. Then $(f^n)$ is not compactly divergent.
    Let $M$ be the 
    limit manifold of $f$. Proceeding similarly as in Theorem~\ref{P:retractofG2}
we conclude that if $\text{dim}(M)=0$ or $1$
then $\text{Fix}(f)$ is either a singleton or equal to $M$;
consequently, $\text{Fix}(f)$ is a holomorphic retract.
Now assume that $\text{dim}(M)=2$. By a result of Ghosh--Zwonek 
\cite[Theorem-4.1]{GargiZwo2025} we conclude that
either $M \cong\mathbb{D}^{2}$ or $M \cong \mathbb{G}_{2}$.
\smallskip

First we consider the case $M \cong\mathbb{D}^{2}$.
Let $\varphi:M \to \D^{2}$ be a biholomorphism,
then $\varphi({\rm Fix (f)})=\text{Fix}(\varphi\circ f\circ\varphi^{-1})$.
Notice that, $\varphi\circ f \circ\varphi^{-1}\in \text{Aut}(\mathbb{D}^{2})$. Hence, $\text{Fix}(\varphi\circ f\circ \varphi^{-1})$
is a holomorphic retraction of $\D^{2}$.
If $r: \D^{2} \to \D^{2}$ is a retraction map for
$\text{Fix}(\varphi \circ f\circ \varphi^{-1})$,
then $\varphi^{-1} \circ r \circ \varphi: M \to M$
is a retraction and ${\rm Fix}(f)$ is the corresponding retract. 
Since, $M $ is retract of $\mathbb{E}$ and ${\rm Fix}(f)$
is retract of $M$, ${\rm Fix}(f)$ is retract of $\mathbb{E}$.
\smallskip

Now suppose $M\cong\mathbb{G}_{2}$. 
In this case, 
$h(\text{Fix}(f))=\text{Fix}(h\circ f\circ h^{-1})$, 
where $h: M \to \mathbb{G}_{2}$ is a biholomorphism. Clearly,  $\text{Fix}(h\circ f\circ h^{-1}) \neq \emptyset.$ 
We now  infer from Theorem~\ref{P:retractofG2} that ${\rm Fix}(h\circ f\circ h^{-1})$ is holomorphic retract of $\GG$. If $r:\GG \to \GG$ is retraction of ${\rm Fix}(h\circ f\circ h^{-1})$ then $h^{-1}\circ r \circ h:M \to M$ is a holomorphic retraction of ${\rm Fix}(f)$ in $M$. Since $M$ is a holomorphic retract of $\mathbb{E}$, ${\rm Fix}(f)$ is a holomorphic retract of $M$, hence,   ${\rm Fix}(f)$ is a holomorphic retract of $\mathbb{E}.$

\smallskip

We now consider
the case when $\text{dim}(M)=3$. In this case, 
$f \in \text{Aut}(\mathbb{E})$. We now claim the following.
\smallskip

\noindent{\bf Claim.} ${\rm Fix}(f) \cap \mathcal{T} \neq \emptyset$, 
where $\mathcal{T}:=\{(x_1, x_2, x_3)\in \mathbb{E}\,:\,x_3=x_1 x_2\}$
is the triangular set. 
\smallskip

\noindent It is a fact that 
that $f(\mathcal{T}) \st \mathcal{T}$.
We also have that $\mathcal{T} \cong \mathbb{D}^{2}$
and $\mathbb{D}^{2}$ is convex. 
Since $\mathbb{D}^{2}$ has the weak Wolff--Denjoy property,
the map $f|_{\mathcal{T}}$ either has a fixed point or 
it is compactly divergent. Therefore, if the map $f|_{\mathcal{T}}$ has
no fixed point in $\mathcal{T}$, then there exists $z_{0} \in \mathcal{T}$
such that $f^{n}(z_{0}) \to \partial \mathcal{T} \st \partial \mathbb{E}$
as $n \to \infty$. Hence, from \cite[Theorem~1.1]{Abate1991}
it follows that $\{f^{n}\}$ is compactly divergent.
Now Theorem~\ref{T:WWD} implies that
$ \text{Fix}(f) = \emptyset $. 
Consequently, we get a contradiction. This established the claim.  
\hfill $\blacktriangleleft$
\smallskip

\noindent
It is  known that $\text{Aut}(\mathbb{E})$ acts
transitively on $\mathcal{T}$ \cite[Remark 6.6]{AWY2007}.
Therefore, there is $\varphi \in \text{Aut}(\mathbb{E})$ such that
the origin is a fixed point of the map  $\varphi\circ f\circ \varphi^{-1}$.
Now it follows from \eqref{E:auttetra} that the
map $\widetilde{f}:=\varphi\circ f\circ \varphi^{-1}$ is either
of the form $L_{\nu}\circ R_{\chi}$ or of the
form $L_{\nu}\circ R_{\chi}\circ F$, 
where $\nu, \chi \in \text{Aut}(\mathbb{D})$ and 
$L_{\nu}, R_{\chi}$, $F$ as \eqref{E:auttetra}. 
\smallskip

First suppose $\widetilde{f}:=L_{\nu}\circ R_{\chi}(z_{1},z_{2},z_{3})$
for some  $\nu, \chi \in {\rm Aut}(\D)$. 
Assume that  $$\nu=\omega \frac{z-\alpha}{\overline{\alpha}z-1}\, ~ {\rm and}~\, \chi=\sigma \frac{z-\beta}{\overline{\beta}z-1}.$$ Since the origin is the fixed point of the map $\tilde{f}$, we have the following relation:
\begin{align}\label{E:E1}
    \begin{bmatrix}
      \omega &   -\omega \alpha \\
      \overline{\alpha}&-1 
\end{bmatrix}\begin{bmatrix} 0 &   0 \\
     0&-1 
\end{bmatrix}\begin{bmatrix}
\sigma &   -\sigma \beta \\
      \overline{\beta}&-1 
    \end{bmatrix}=\begin{bmatrix}
    0 &  0 \\
    0&-1 
    \end{bmatrix}
\end{align}

\noindent
From \eqref{E:E1}, obtain $\alpha =\beta =0$, and
$\widetilde{f}(z_1, z_2, z_3)=(-\omega z_1, -\sigma z_2, \sigma\omega z_3)$. 
If $(z_{1}, z_{2}, z_{3}) \in \text{Fix}(\tilde{f})$ then we have that $\sigma\omega z_{3}=z_{3}$, $\sigma z_{2}=-z_{2}$ and $\omega z_{1}=-z_{1}$. Consider the following subcases:

\smallskip

\noindent
\textbf{Case~1.} $z_{3} \neq 0$. 
\smallskip

\noindent Clearly $\sigma \omega = 1$. Consequently, the map $\tilde{f}(z_{1},z_{2},z_{3})=(-\omega z_{1},-\overline{\omega}z_{2}, z_{3})$. Now if there exists $(z_{1},z_{2},z_{3}) \in \text{Fix}(f)$ with $z_{1} \neq 0$ or $z_{2}\neq 0$ then we have $\omega=-1$. Then we have that $\tilde{f}= {\rm I_{\mathbb{E}}}$. Consequently, $\text{Fix}(f)=\mathbb{E}$, a trivial retract. If this is not the case, then we have that $\text{Fix}(f)=\{(z_{1},z_{2},z_{3}) \in \mathbb{E}:z_{1}=z_{2}=0\}$. Then it follows that $\text{Fix}(f)$ is a one-dimensional retract, and the retraction map is $P_{3}:\mathbb{E} \to \mathbb{E}$ defined by $P_{3}(z_{1},z_{2},z_{3})=(0,0,z_{3})$.
\smallskip

\noindent\textbf{Case~2.} $z_{3}=0$.

\noindent In this case if $(z_{1},z_{2},z_{3}) \in \text{Fix}(f) $ such that $z_{1}, z_{2} \neq 0$,
then $\tilde{f}(z) = {\rm I_{\mathbb{E}}}.$ Consequently, $\text{Fix}(f)=\mathbb{E}$, a trivial retract. On the other hand if either $z_{1}$ or $z_{2}$ is equal to $0$ then it follows that $\text{Fix}(\tilde{f}) \subset \mathcal{T}$.
In this situation, we can consider the map $\tilde{f}|_{\mathcal{T}}: \mathcal{T} \to \mathcal{T}$. Since, $\mathcal{T} \cong \mathbb{D}^{2}$, hence, $\text{Fix}(\tilde{f})$ is a holomorphic retract. 
\smallskip

We now consider  the  other case, when $\tilde{f}(z)=L_{\nu}\circ R_{\chi}\circ F(z_{1},z_{2},z_{3})=L_{\nu}\circ R_{\chi}(z_{2}, z_{1}, z_{3})$.
Proceeding exactly as in the previous argument, we deduce that the condition $\tilde{f}(0)=0$ implies that  $\tilde{f}(z_{1},z_{2},z_{3})=(-\omega z_{2}, -\sigma z_{1}, \sigma \omega z_{3})$. If $(z_{1},z_{2},z_{3}) \in {\rm Fix}(\mathbb{E})$, then we have the following relation:
\begin{align}\label{E:eqfix}
    -\om z_{2}=z_{1},\quad~ -\sigma z_{1}=z_{2},\quad ~~\sigma \om z_{3}=z_{3}.
\end{align}
If $\sigma\omega\neq 1$, then it follows from the above relation that $z_{3}=z_{2}=z_{1}=0$.   Consequently, it follows that ${\rm Fix}(f)$ is a singleton set, hence, a holomorphic retract. 

\smallskip
\noindent
If $\sigma\om=1$, then it follows from the relation \eqref{E:eqfix} that 
$${\rm Fix}(f)=\{(z_{1},z_{2},z_{3}) \in \mathbb{E}: z_{1}=-\om z_{2}\}.$$

\noindent
We  now consider $\Phi \in {\rm Aut}(\mathbb{E})$ defined by $\Phi\big(x_{1},x_{2},x_{3}\big)=\big(\eta x_{1}, \eta^{-1}x_{2},x_{3}\big)$ where $\eta^{2}=-\overline{\om}$. Then $\Phi({\rm Fix}(f))=\{(z_{1},z_{2},z_{3}) \in \mathbb{E}: z_{1}=z_{2}\}.$  Note that, the map  $R:\mathbb{E}\to \mathbb{E}$ defined by
$$R(x_1,x_2,x_3)
=
\left(
\frac{x_1+x_2}{2},
\frac{x_1+x_2}{2},
x_3
\right) $$ is a retraction map of $\Phi({\rm Fix}(f))$. Therefore, we obtain that $\Phi({\rm Fix}(f))$ is a holomorphic retract and, consequently, ${\rm Fix}(f)$ is also a holomorphic retract. Therefore,   we conclude the theorem.

\end{proof}
\smallskip

\subsection{Proof of Theorem~\ref{T:limit in G2}}\label{SS:limitsetsymbidisk}
We now present the proof of Theorem~\ref{T:limit in G2}. 

\begin{proof}[Proof of Theorem~\ref{T:limit in G2}]
Let $\xi \in T(f)$. Hence, by definition,
 there exists $(s,p)\in \GG$ and
a sequence $(n_{k})_{k\in \mathbb{N}}$ such that
$f^{n_{k}}(s,p)\to\xi\in\partial\GG$ as $k\to\infty$.
Note that if we let $z_{0}=(s_{0},p_{0})$, then  
$f^{n_{k}}(s_{0},p_{0})$ is a bounded sequence,
hence, there exists a subsequence $n_{k_{j}}$ such that $f^{n_{k_{j}}}(s_{0},p_{0})\to\zeta\in\partial \GG$. 
By our  assumption   $\zeta=(2e^{i\theta_{0}}, e^{2i\theta_{0}})$
for some $\theta_{0}\in [0, 2\pi)$.
We now consider the function $h:\GG\to\overline{\D}$ defined by
\[
h(s,p)=\dfrac{se^{-i\theta_{0}}}{2}.
\]
Note that  the map $h$ is defined on $\CC$, 
$h(\zeta)=1$ and $|h(s,p)|<1$ for all $(s,p)\in\GG$.
We now consider a holomorphic peak function of $\D$,
defined as follows: 
$\rho_{\om}:\overline{\D}\to\overline{\D}$ 
where $\om\in\mathbb{T}$ and $\rho_{\om} \in C(\overline{\D})\cap \rm Hol(\D)$ with $\rho_{\om}(\overline{\D}\setminus\{\om\})\subset\D$, $\rho_\omega(\om)=1$.
Let us consider the sequence of function $g^{n_{k_{j}}}:\GG\to \overline{\D}$ 
defined by $g^{n_{k_{j}}}(s,p)=\rho_{1}(h(f^{n_{k_{j}}}(s,p)))$. By Montel's theorem, there exists $g\in\rm Hol(\GG,\overline{\D})$ 
such that $g^{n_{k_{j}}}(s,p)\to g(s,p)$ as $j\to\infty$ (upto a subsequence) 
locally uniformly on $\GG$. Note that,
$g(s_{0},p_{0})=\rho_{1}(h(\zeta))=1$. Then by the maximum modulus principle,
we deduce that $g(s,p)= 1$ for all $(s,p) \in \GG$.
From this observation, we deduce that 
if $\xi=(a_{1}+a_{2},a_{1}a_{2})$, then $h(a_{1}+a_{2}, a_{1}a_{2})=1$,
so $(a_{1}+a_{2})= 2e^{i\theta_{0}}$. Clearly, $|a_{1}+a_{2}|=2$.
Hence, it follows that $a_{1}=a_{2}=e^{i\theta}$ with 
$\theta=\theta_{0}(\text{mod} 2\pi)$.
Hence, $\xi=(2e^{i\theta_{0}},e^{2i\theta_{0}})=\zeta$.
This completes the proof of $(i)$.

\smallskip
\noindent

To prove $(ii)$, let $(s,p)\in\GG$  and $\eta\in T(f,(s,p))$.
Clearly, there exists a sequence $(n_{k})_{k\in\mathbb{N}}$ such that $f^{n_{k}}(s,p)\to\eta\in\partial \GG$.
Note that, $f^{n_{k}}(s_{0},p_{0})$ is a bounded sequence.
Hence, there is a subsequence $n_{k_{l}}$ such that $f^{n_{k_{l}}}(s_{0},p_{0}) \to\xi\in\partial\GG$.
By our assumption,
$\xi=(a+e^{i\theta_{0}}, ae^{i\theta_{0}})$ for some $a$ with $|a|<1$.
Consider the function $\Phi_{\om}:\GG\to\cplx$
defined by $$\Phi_{\omega}(s,p):=\frac{2\overline{\omega}p-s}{2-\overline{\omega}s},~~\omega\in\mathbb{T}.$$ It follows from Result~\ref{R:AglerYoung1} that $\Phi_{\om}(\GG) \subset \D$.
Note that  $\Phi_{e^{i\theta_{0}}}$ has continuous extension
through the boundary point $\xi$ and
 $|\Phi_{e^{i\theta_{0}}}(\xi)|=|-e^{i\theta_{0}}|=1$.
 Consider the  sequence of functions $h_{l}:\GG\to\D$
 defined by 
 \[
 h_{l}(s,p):=\rho_{-e^{i\theta_{0}}}\circ \Phi_{e^{i\theta_{0}}}\circ f^{n_{k_{l}}}(s,p)
 \]
 where the function
 $\rho_{-e^{i\theta_{0}}}$ is defined as in the proof of $(i)$.
 By Montel's theorem we deduce that there exists a subsequence $l_{m}$
 such that $h_{l_{m}}(s,p) \to h(s,p)$ for some holomorphic map
 $h:\GG\to \overline{\D}$. Note that  
\begin{align*}
h(s_{0},p_{0})&=\lim_{m \to \infty}h_{l_{m}}(s_{0},p_{0})\\
&=\lim_{m \to \infty}\rho_{-e^{i\theta_{0}}}\circ \Phi_{e^{i\theta_{0}}}\circ f^{n_{k_{l_m}}}(s_0,p_0)=\rho_{-e^{i\theta_{0}}}(\Phi_{e^{i\theta}}(\xi))=1.
\end{align*} 
Therefore, by the maximum modulus principle,
 $h(z)= 1$ for all $z \in \GG$. Let $\eta=(a_{1}+a_{2}, a_{1}a_{2})$
 with $(a_{1},a_{2}) \in \overline{\D} \times \overline{\mathbb{D}}$.
 If $|a_{1}+a_{2}|=2$
 then we have $\eta=(2\omega,\omega^{2})$
 for some $\omega\in\mathbb{T}$, then from $(i)$
 we conclude that $T(f,z_{0}) \subset \partial_{s}\GG$.
 This contradicts the hypothesis. Hence, $|a_{1}+a_{2}|<2$.
 Note that the map $\Phi_{e^{i\theta}}$ has continuous extension
 through the boundary point $\eta$. Let $\eta=(e^{i\theta_{1}}+e^{i\theta_{2}}, e^{i\theta_{1}}e^{i\theta_{2}})$
 with $\theta_{1}\neq \theta_{2} (\text{mod}~2\pi)$.
 Since the map $h(s,p)=1$,
 $\rho_{-e^{i\theta_0}}(\Phi_{e^{i\theta_0}}(e^{i\theta_{1}}+e^{i\theta_{2}}, e^{i\theta_{1}}e^{i\theta_{2}}))=1$.
Clearly it follows that 
$\Phi_{e^{i\theta_0}}(e^{i\theta_{1}}+e^{i\theta_{2}},
e^{i\theta_{1}}e^{i\theta_{2}})=-e^{i\theta_0}$.
A simple computation gives us 
 \begin{align*}
(e^{i\theta_0}-e^{i\theta_{1}})(e^{i\theta_0}-e^{i\theta_{2}})&=0
\end{align*}
Therefore, either $\theta_{1}=\theta_0 (\text{mod}2\pi)$
or $\theta_{2}=\theta_0 (\text{mod}~2\pi)$. 
Hence, $T(f) \subset \pi (\overline{\D}\times e^{i\theta}) \cup \pi ( e^{i\theta} \times \overline{\D})$.
\end{proof}
We conclude this article with the following remark.
\begin{remark}
Theorem~\ref{T:Tichotomy} does not extend to $\mathbb{C}^n$ for $n \geq 4$ via our 
method, since Result~\ref{R:Abate} is not applicable in higher dimensions.
In \cite{Gargi2}, Ghosh and Zwonek introduced the class $\mathbb{L}_n$,
defined as the image of a two-proper holomorphic map from a Lie ball $L_n$. In particular, $\mathbb{G}_2 \cong \mathbb{L}_2$ and $\mathbb{E} \cong \mathbb{L}_3$, so $\mathbb{L}_n$ has the weak Wolff--Denjoy property for $n=2,3$. Moreover, by \cite[Theorem~5.3]{Gargi2}, every $f \in \mathrm{Aut}(\mathbb{L}_n)$ has this property. This leads to the question whether $\mathbb{L}_n$ enjoys the weak Wolff--Denjoy property for all $n \geq 4$.
\smallskip

For $\mathbb{L}_4$, let $f \in \mathrm{Hol}(\mathbb{L}_4,\mathbb{L}_4)$ with ${\rm Fix}(f)=\emptyset$, and suppose $(f^n)$ is not compactly divergent. Arguing as in Theorem~\ref{T:WWD}, consider ${\rm dim}(M) \in \{0,1,2,3,4\}$, 
where $M$ is the limit manifold of $f$. If ${\rm dim}(M)=4$, then $f \in \mathrm{Aut}(\mathbb{L}_4)$, which implies $\{f^n\}$ is compactly divergent—a contradiction. The cases ${\rm dim}(M)\in\{0,1,2\}$ follow as before, leaving only ${\rm dim}(M)=3$, with $M$ passing through the origin.
A complete classification of three-dimensional holomorphic retracts of $\mathbb{L}_4$ through the origin is currently unknown. If, up to biholomorphism, the only such retracts are $\{0\}\times L_3$ and $\mathbb{L}_3 \times \{0\}$, then Theorem~\ref{T:Tichotomy} would imply that $f$ is a periodic automorphism of one of these domains. Since every such automorphism has a fixed point, it would follow that $\mathbb{L}_4$ has the weak Wolff--Denjoy property.
\end{remark}
\noindent {\bf Acknowledgements.} 
We are thankful to Professor Zwonek for pointing out a gap in the statement and
in the proof of Theorem~\ref{P:retractofG2} in an earlier version of this article.  
S. Chatterjee is supported by the Institute Postdoctoral Fellowship of the Indian Institute of Technology, Kanpur. C. Sur is supported by a CSIR fellowship (File
No-09/0092(15100)/2022-EMR-I.)

    
\end{document}